\documentclass[11pt,reqno]{amsart}
\usepackage{graphicx,tikz}
\usepackage{amsfonts,amsmath,amssymb}
\usetikzlibrary{matrix,arrows}
\newcommand{\C}{\mathbb C}
\newcommand{\R}{\mathbb R}
\newcommand{\N}{\mathbb N}

\newtheorem{theorem}{Theorem}[section]
\newtheorem{lemma}[theorem]{Lemma}
\newtheorem{prop}[theorem]{Proposition}

\newtheorem{defn}[theorem]{Definition}
\newtheorem{rem}[theorem]{Remark}

\begin{document}

\title {Complex associated to some systems of PDE}

\author{ Pierre Bonneau*  and Emmanuel Mazzilli**}

\maketitle

\begin{abstract}  In \cite{WW1} and \cite{WW2}, the author constructed the Complex associated to $1$-regular functions. This complex is the equivalent of Dolbeault's complex for holomorphic functions if we replace the Cauchy-Riemann equations by the Cauchy-Fueter equations. In this paper, using the Cartan theory of linear Pfaffian system, we give a direct construction for the Cauchy-Fueter complex, at least in $\R^8$. Moreover, we give a sufficient condition in terms of Cartan's theory, to ensure that a complex associated to a linear PDE system with constant coefficients of order one, contains only operators of order one. In fact, the Cauchy-Fueter equation in $\R^8$ is an illuminating example for which this condition is not satisfied. \end{abstract}

\section{introduction}
\bigskip

The aim of this paper is to give a complete construction of the complex associated to the 1-regular functions. This complex was first obtained by Wang Wei in \cite{WW1} and \cite{WW2} using classical theory of Leray's spectral sequences. Here we give a more elementary construction using the Cartan theory of involution for linear Pfaffian exterior differential system. For simplicity, we restrict ourselves to germs of $1$-regular functions defined in the neighborood of a point $z\in\C^8$ with value in $\C^2$. Using the notation of \cite{WW2}, the coordinates on $\C^8$ will be  $(z^{ij})$ with $0\leq i\leq 3$ and $0\leq j\leq 1$. The $1$-regular functions are the solutions of the following PDE system 
\begin{equation}
{\partial \phi^{0}\over z^{i0}}+{\partial\phi^{1}\over z^{i1}}=0,
\end{equation}
for all $i$. Nevertheless the previous homogeneous system is overdetermined, $\phi$ identically equal zero, is not the only solution (it will be clear in the following). We will explain in more details the  construction of the first step in the $1$-regular complex. We have to solve the non homogeneous system associated to the first one 
\begin{equation}{\partial \phi^{0}\over z^{i0}}+{\partial\phi^{1}\over z^{i1}}=\Phi^{i}\ \ ,(J).
\end{equation}
The system is overdetermined and so it has no solution for all $(\Phi^i)$. The first difficulty is to find the constraint of integration; if we made an analogy with the Dolbeault complex for $\bar\partial$ and we use the language of differential form, the equivalent of the system $J$ in this setting is the system 
\begin{equation}\bar\partial f=u,
\end{equation}
with $u$ a $1$-form in $\C^8$ and a function $f$ in $\C^8$ with values in $\C$. It is well known that we have two constraints to solve these equations : $u$ has to be a $(0,1)$-form and $\bar\partial u=0$. In our case, it is a little more difficult to find the constraints and the solutions. The system $J$ can be viewed as a linear pfaffian exterior differential system so we can apply the Cartan's theory to find solutions of it. More precisely, if we introduce $p^{a}$, $p^{a}_{ij}$ as free variables for $\Phi^{a}$, ${\partial\phi^{a}\over z^{ij}}$, the structure forms of the  Pfaffian system associated to $J$ are, 
$$dp^{a}-p^{a}_{ij}dz^{ij},$$ 
on the manifold $M$ defined by the equations: $$p^{0}_{i0}+p^{1}_{i1}=\Phi^{i},\ \ (J)$$ 
with independence condition $\wedge_{i,j}dz^{ij}\not = 0$ on $M$. Roughly speaking the Cartan theory implies existence of solution of $J$ passing through all points of $M$ if the system is in involution. For Pfaffian linear system the involution is equivalent of the two following things : the torsion is zero for all points in $M$ and the tableau associated to $J$, $A_x$, must be involutive for all $x$ in $M$ (see \cite{BCGGG} for the definition of $A_x$). The torsion is the obstruction of the existence of solutions of the first prolongation system of $J$, denoted by $J^{1}$, in sense of jets. With the variables $p^{a}_{ijkl}$ for ${\partial^{2} \phi^{a}\over \partial z^{ij}\partial z^{kl}}$, the first prolongation system $J^{1}$ is the linear Pfaffian system with structure forms
$$dp^{a}-p^{a}_{ij}dz^{ij},$$
$$dp^{a}_{ij}-p^{a}_{ijkl}dz^{kl},$$
on the manifold $M^{1}$ defined by the equations:
$$p^{0}_{i0}+p^{1}_{i1}=\Phi^{i},$$
$$p^{0}_{i0kl}+p^{1}_{i1kl}={\partial\Phi^{i}\over \partial z^{kl}}.$$ Clearly the two systems have exactly the same solutions. Finally the torsion is the compatibility conditions for which the last two linear equations have solutions with $p^{a}_{ijkl}$ symmetric by change of pairs $ij$ and $kl$. We can define by induction the prolongation $J^{q}$ of $J$ on the manifold $M^{q}$ for all $q\in \N$. Using the proposition 3.9 of \cite{BCGGG}, there exists $k_0$ such that for all $k\geq k_0$ the tableau $A^{q}$ associated to $J^{q}$ is in involution. To obtain a solution of $J$, it is sufficient to compute the torsion of all $J^{q}$, which is done precisely in the  section $2$ for more general system than the Cauchy-Fueter equations  with intrinsic definition of the torsion. In section $3$, we calculate the torsion of the PDE system induced by the torsion of $J$ by means of Spencer's cohomology. This torsion will give the second linear operator in the complex associated to the $1$-regular functions, and so on, until we obtain a system without torsion (see section $4$). Furthermore, in the section $4$, we give a sufficient condition for a linear PDE system with constant coefficients of order one, to have only first order operators in the associated complex, in terms of "tableau"'s involutivity, which seems new.

\section{Some systems of PDE}
Let us consider the system of partial differential equations 
\begin{equation}\label{W2}
\frac{\partial \phi^{0}}{\partial z^{j_{0}0}}+\frac{\partial \phi^{1}}{\partial z^{j_{1}1}}=\varphi^{j_{0},j_{1}},\quad with \;\;(j_{0},j_{1})\in \mathfrak{I}\subset \{1,...,n\}\times \{1,...,m\},
\end{equation}
where the unknown functions $\phi^{0},\;\phi^{1}$ are complex functions defined in an open set $\Omega$ of $\C^{n+m}=\C^{n}\times\C^{m}$ with coordonates $(z^{10},\;z^{20},...,z^{n0})\in \C^{n},$ and $(z^{11},\;z^{21},...,\;z^{m1})\in \C^{m},$ and the functions $\varphi^{j_{0},j_{1}}$ given in the second member are complex functions defined in 
$\Omega.$ \\
$\mathfrak{I}$ is a subset of  $\{1,...,n\}\times \{1,...,m\},$ and the system \ref{W2} is a system with $card(\mathfrak{I})$ equations. \\
We denote $\mathfrak{I}^{1}=\{j\in \{1,...,m\}:\exists j_{0}\in \{1,...,n\}:(j_{0},j)\in \mathfrak{I}\},\;\;$ $\mathfrak{I}^{0}=\{i\in \{1,...,n\}:\exists i_{1}\in \{1,...,m\}:(i,i_{1})\in \mathfrak{I}\}, $ and $\overline{\mathfrak{I}^{1}}=\{1,...,m\}-\mathfrak{I}^{1},\;\;\overline{\mathfrak{I}^{0}}=\{1,...,n\}-\mathfrak{I}^{0}.$\\
Suppose that $(j_{0},j)$ and $(j'_{0},j)$ are two elements of $\mathfrak{I}.$ Then the system \ref{W2} contains the two equations 
\begin{equation}\label{W3}
\frac{\partial \phi^{0}}{\partial z^{j_{0}0}}+\frac{\partial \phi^{1}}{\partial z^{j1}}=\varphi^{j_{0},j},
\end{equation}
and
\begin{equation}\label{W1}
\frac{\partial \phi^{0}}{\partial z^{j'_{0}0}}+\frac{\partial \phi^{1}}{\partial z^{j1}}=\varphi^{j'_{0},j}.
\end{equation}
Therefore, 
\begin{equation}
\frac{\partial \phi^{1}}{\partial z^{j1}}=\varphi^{j_{0},j}-\frac{\partial \phi^{0}}{\partial z^{j_{0}0}}=\varphi^{j'_{0},j}-\frac{\partial \phi^{0}}{\partial z^{j'_{0}0}},
\end{equation}
and, consequently, 
\begin{equation}
\varphi^{j'_{0},j}=\varphi^{j_{0},j}+\frac{\partial \phi^{0}}{\partial z^{j'_{0}0}}-\frac{\partial \phi^{0}}{\partial z^{j_{0}0}}.
\end{equation}
Reporting this expression of $\varphi^{j'_{0},j}$ in \ref{W1} gives \ref{W3}. We have then two times the same equation. So, we suppose that if $(j_{0},j)$ and $(j'_{0},j)$ are in $\mathfrak{I},$ then $j'_{0}=j_{0}.$ Consequently, to one element $j$ of $\mathfrak{I}^{1}$ corresponds one element $j_{0}$ of $\mathfrak{I}^{0}.$ Similarly, to one element $i$ of $\mathfrak{I}^{0}$ corresponds one element $i_{1}$ of $\mathfrak{I}^{1}.$ We have a bijection $b$ between $\mathfrak{I}^{1}$ and $\mathfrak{I^{0}}.$ Often, we shall denote $j_{0}=b(j)$ when $j\in \mathfrak{I}^{1}.$ The number of equations in \ref{W2} is $card(\mathfrak{I}^{1})=card(\mathfrak{I}^{0})=card(\mathfrak{I}).$\\
According with \cite{BCGGG}, we note \\
$p^{k}_{i0}=\frac{\partial \phi^{k}}{\partial z^{i0}},$  $p^{k}_{i0,j1}=\frac{\partial^{2} \phi^{k}}{\partial z^{i0}\partial z^{j1}},$ $p^{k}_{i0,j1,l1}=\frac{\partial^{3} \phi^{k}}{\partial z^{i0}\partial z^{j1}\partial z^{l1}},$ $\varphi^{j_{0},j}_{i0}=\frac{\partial \varphi^{j_{0},j}}{\partial z^{i0}}$ and so on. We also note $\varphi^{j}=\varphi^{j_{0},j}.$\\
Then, \ref{W2} can be written 
\begin{equation}\label{W4}
p^{0}_{j_{0}0}+p^{1}_{j1}=\varphi^{j},\quad j\in \mathfrak{I}^{1}.
\end{equation}
We want to obtain the torsion of the system \ref{W4} by using the notations and results of \cite{BCGGG}. \\
If $I=\{i_{1},i_{2},...,i_{r} \}$ with $i_{k}\in \{1,...,n\}$ and $J=\{j_{1},...,j_{s}\}$ with $j_{k}\in \{ 1,...,m \}$  are multi-indices, we note $$I+i_{k}=I,i_{k}=\{i_{1},i_{2},...,i_{r},i_{k}\},\quad I-i_{k}=\{i_{1},i_{2},...,i_{k-1},i_{k+1},...,i_{r}\},$$
$$I0=\{i_{1}0,i_{2}0,...,i_{r}0\},
\quad I1=\{i_{1}1,i_{2}1,...,i_{r}1\},\quad p^{a}_{I0,J1}=\frac{\partial^{r+s} \phi^{a}}{\partial z^{i_{1}0}\partial z^{i_{2}0}...\partial z^{i_{r}0}\partial z^{j_{1}1}\partial z^{j_{2}1}...\partial z^{j_{s}1}}.$$
We want to look for the torsion of any order of the system \ref{W4} by using the methods and notations of 
\cite{BCGGG}. \\
We have, if we now note $I=\{i_{1}0,...,i_{r}0,j_{1}1,...,j_{s}1\}$ with $i_{k}\in \{1,...,n\}$ and $j_{l}\in\{1,...,m\},$ by deriving \ref{W4}
\begin{equation}\label{W8}
p^{1}_{I,j1}=\varphi^{j}_{I}-p^{0}_{I,j_{0}0}\quad when \;\; (j_{0},j)\in \mathfrak{I},
\end{equation}
and obtain the structure equations 
\begin{equation}\label{W5}
\begin{cases}
\theta^{0}:=d\phi^{0}-p^{0}_{i_{0}}dz^{i0}-p^{0}_{j1}dz^{j1}=0\\
\theta^{0,i0}:=dp^{0}_{i0}-p^{0}_{i_{0},i'0}dz^{i'0}-p^{0}_{i0,j'1}dz^{j'1}=0\\
\theta^{0,J}:=dp^{0}_{J}-p^{0}_{J,i'_{0}}dz^{i'0}-p^{0}_{J,j'1}dz^{j'1}=0\quad when\;\;J\subset I\\
\theta^{1}:=d\phi^{1}-p^{1}_{i_{0}}dz^{i0}-p^{1}_{j1}dz^{j1}=0\\
\theta^{1,i0}:=dp^{1}_{i0}-p^{1}_{i_{0},i'0}dz^{i'0}-p^{1}_{i0,j'1}dz^{j'1}=0\\
\theta^{1,J}:=dp^{1}_{J}-p^{1}_{J,i'0}dz^{i'0}-p^{1}_{J,j'1}dz^{j'1}=0\quad when\;\;J\subset I.\\
\end{cases}
\end{equation}
If $M$ is the variety in the space of the variables $\phi^{0},\;\phi^{1},\;z^{10},...,z^{n0};\;z^{11},...,z^{m1},\;p^{0}_{J},\;p^{1}_{J}\;\;(with\;\;J\subset I),\;p^{0}_{I,i'0}, \;p^{0}_{I,j'1},\;P^{1}_{I,i'0},\;p^{1}_{I,j'1},$ defined by the conditions \ref{W4}, we can consider the cotangent space $T^{*}M$ which is generated by 
\begin{equation}
\begin{split}
&\theta^{0},\;\theta^{1},\;\theta^{0,J},\;\theta^{1,J},\;\;(with\;\;J\subset I),\;dz^{10},...,dz^{n0},\;dz^{11},...,dz^{m1},\;dp^{0}_{J},\;dp^{1}_{J},\;\\
 &dp^{0}_{I,i'0},\;dp^{0}_{I,j'1},\;dp^{1}_{I,i'_{0}},\;\; and\;\; (when \;\;j'\in \overline{\mathfrak{I}^{1}}),\;\;dp^{1}_{I,j'1}.
\end{split}
\end{equation}
When $j'\in \mathfrak{I}^{1},$ according with \ref{W4}, $dp^{1}_{I,j'1}$ is replaced by $
d\varphi^{j'}_{I}-dp^{0}_{I,j'_{0}0}.$\\
From \ref{W5}, we can deduce
\begin{equation}\label{W6}
\begin{cases}
-d\theta^{0,J}&=dp^{0}_{J,i'0}dz^{i'0}+dp^{0}_{J,j'1}dz^{j'1}\quad with\;\;J\subset I\\
-d\theta^{1,J}&=dp^{1}_{J,i'0}dz^{i'0}+dp^{1}_{J,j'1}dz^{j'1}\quad with\;\;J\subset I,\quad J\neq I\\
-d\theta^{1,I}&=dp^{1}_{I,i'0}dz^{i'0}+dp^{1}_{I,j'1}dz^{j'1}\\
&=dp^{1}_{I,i'0}dz^{i'0}+\Sigma_{j'\in \mathfrak{I}^{1}}(d\varphi^{j'}_{I}-dp^{0}_{j'_{0}0,I})dz^{j'1}+\Sigma_{j'\in \overline{\mathfrak{I}^{1}}}dp^{1}_{I,j'1}\\
&=dp^{1}_{I,i'0}dz^{i'0}-\Sigma_{j'\in\mathfrak{I}^{1}}dp^{0}_{I,j'_{0}0}dz^{j'1}
+\Sigma_{j'\in\overline{\mathfrak{I}^{1}}}dp^{1}_{I,j'1}
+\Sigma_{j'\in\mathfrak{I}^{1}}\varphi^{j'}_{I,i0}dz^{i0}dz^{j'1}\\
&\qquad\qquad+\Sigma_{j'\in\mathfrak{I}^{1}}\varphi^{j'}_{I,j1}dz^{j1}dz^{j'1}.
\end{cases}
\end{equation}
We precise the summation domain under the sign $\Sigma,$ except when the dommain of summation concern all the indices, in this case, conformly with the Einstein convention, the indices are only repeated. \\
We have to compare our notations, inspired by \cite{WW2}, with those of \cite{BCGGG}. $\theta^{a},$ in \cite{BCGGG} page 129, is indexed by $a$ or $b$ so, here, we have $a=(0,J)$ with $J\subset I$ (possibly $\emptyset$). The variables are indexed by $i$ or $j,$ and, now, we have $i=i0\;\;or\;\;j1$ with $i\in \{1,...,n\}$ and $j\in \{1,...,m\}.$ And the terms $dp$ (noted $\pi$ in page 129 of \cite{BCGGG}) are indexed by $\varepsilon$ or $\delta$ which now becomes \\
$\varepsilon = (0,J)\;\;or\;\;(0,I,i'_{0})\;\;or\;\;(0,I,j'1)\;\;or\;\; (1,J)\;\;or\;\;(1,I,i'_{0})\;\;or\;\;(1,I,j'1).$\\
So, translating the formulas of \cite{BCGGG} in page 130, we have, from \ref{W6}, if $J\neq I,$
\begin{equation}\label{W7}
\begin{split}
A^{0,J}_{(0,J,i'0),i0}=A^{1,J}_{(1,J,i'0),i0}=\delta_{i}^{^{i'}},\quad A^{0,J}_{(0,J,j'1),j1}=A^{1,J}_{(1,J,j'1),j1}=\delta_{j}^{^{j'}},\\ 
c^{0,J}_{i0,i'0}=c^{0,J}_{j1,j'1}=c^{0,J}_{i0,j1}=c^{1,J}_{i0,i'0}=c^{1,J}_{j1,j'1}=c^{1,J}_{i0,j1}=0,
\end{split}
\end{equation}
all the others expressions $A^{0,J}_{.,.}$ being $0. $ And, when $J=I,$ we have
\begin{equation}\label{W7b}
\begin{split}
&A^{0,I}_{(0,I,i'0),i0}=\delta_{i}^{^{i'}},\quad A^{0,I}_{(0,I,j'1),j1}=\delta_{j}^{^{j'}},\\ 
&A^{1,I}_{(1,I,i'0),i0}=\delta_{i}^{^{i'}},\quad A^{1,I}_{(1,I,j'1),j1}=\delta_{j}^{^{j'}}\;\; if\;\;j\in\overline{\mathfrak{I}^{1}},\\ 
&A^{1,I}_{(0,I,j_{0}0),j1}=-1\;\; if\;\;j\in\mathfrak{I}^{1}\;\;and\;\; 0\;\; else,\\
&c^{0,I}_{i0,i'0}=c^{0,i}_{j1,j'1}=c^{0,I}_{i0,j1}=0\\
&c^{1,I}_{i0,i'0}=c^{1,I}_{j1,j'1}=c^{1,I}_{i0,j1}=0\;\; if\;\; j\in\overline{\mathfrak{I}^{1}},\\
&c^{1,I}_{i0,j1}=\frac{\partial \varphi^{j}_{I}}{\partial z^{i0}}=\varphi^{j}_{I,i0}\;\; and\;\;
c^{1,I}_{j'1,j1}=\frac{\partial \varphi^{j}_{I}}{\partial z^{j'1}}=\varphi^{j}_{I,j'1}\;\; if\;\;j\in\mathfrak{I}^{1}.
\end{split}
\end{equation}

Let $\mathcal{I}\subset \mathcal{J}\subset T^{*}M$ be a filtration of $T^{*}M$ like that of \cite{BCGGG} page 129, that is to say $\mathcal{I}$ is generated by $\theta^{0,J},\;\;\theta^{1,J}$ with $J\subset I;$ $\mathcal{J}$ is generated by $\theta^{0,J},\;\;\theta^{1,J},\;\;dz^{10},...,dz^{n0},\;\;dz^{11},...,dz^{m1};$ and the generators of $T^{*}M$ are given before.\\
If $p$ is an element of $\mathcal{J}^{\perp}\otimes \mathcal{J}/\mathcal{I},$ (see \cite{BCGGG} page 138), i.e. \\
$p=p^{(0,J)}_{i0}\frac{\partial }{\partial p^{(0,J)}}\otimes dz^{i0}+p^{(0,J)}_{j1}\frac{\partial }{\partial p^{(0,J)}}\otimes dz^{j1}+p^{(1,J)}_{i0}\frac{\partial }{\partial p^{(1,J)}}\otimes dz^{i0}+p^{(1,J)}_{j1}\frac{\partial }{\partial p^{(1,J)}}\otimes dz^{j1}+p^{(1,I,i'0)}_{i0}\frac{\partial }{\partial p^{(1,I,i'0)}}\otimes dz^{i0}+p^{(1,I,i'0)}_{j1}\frac{\partial }{\partial p^{(1,I,i'0)}}\otimes dz^{j1}+\Sigma_{j\in \overline{\mathfrak{I}^{1}}}\Big(p^{(1,I,j'1)}_{i0}\frac{\partial }{\partial p^{(1,I,j'1)}}\otimes dz^{i0}
+p^{(1,I,j'1)}_{j1}\frac{\partial }{\partial p^{(1,I,j'1)}}\otimes dz^{j1}\Big),$\\
we want, using the values $A^{\bullet}_{\bullet,\bullet}$ given in \ref{W7} and \ref{W7b}, to calculate $\overline{\pi}(p)$ (see page 138) and obtain 
\begin{equation}
\begin{split}
&\overline{\pi}(p)=\Sigma_{J\subset I}\big(p^{(0,J,i'0)}_{i0}-p^{(0,J,i0)}_{i'0}\big)\frac{\partial }{\partial \theta^{0,J}}\otimes dz^{i0}\wedge dz^{i'0}+2\big(p^{(0,J,j'1)}_{i0}-p^{(0,J,i0)}_{j'1}\big)\frac{\partial }{\partial \theta^{0,J}}\otimes dz^{i0}\wedge dz^{j'1}\\
&+\big(p^{(0,J,j'1)}_{j1}-p^{(0,J,j1)}_{j'1}\big)\frac{\partial }{\partial \theta^{0,J}}\otimes dz^{j1}\wedge dz^{j'1}\\
&+\Sigma_{J\subset I,\;\; J\neq I}\big(p^{(1,J,i'0)}_{i0}-p^{(1,J,i0)}_{i'0}\big)\frac{\partial }{\partial \theta^{1,J}}\otimes dz^{i0}\wedge dz^{i'0}\\
&+2\big(p^{(1,J,j1)}_{i0}-p^{(1,J,i0)}_{j1}\big)\frac{\partial }{\partial \theta^{1,J}}\otimes dz^{i0}\wedge dz^{j1}
+\big(p^{(1,J,j'1)}_{j1}-p^{(1,J,j1)}_{j'1}\big)\frac{\partial }{\partial \theta^{1,J}}\otimes dz^{j1}\wedge dz^{j'}\\
&+\big(p^{(1,I,i'0)}_{i0}-p^{(1,I,i0)}_{i'0}\big)\frac{\partial }{\partial \theta^{1,I}}\otimes dz^{i0}\wedge dz^{i'0}+2\Sigma_{j\in \overline{\mathfrak{I}^{1}}}\big(p^{(1,I,j1)}_{i0}-p^{(1,I,i0)}_{j1}\big)\frac{\partial }{\partial \theta^{1,I}}\otimes dz^{i0}\wedge dz^{j1}\\
&-2\Sigma_{j\in \mathfrak{I}^{1}}\big(p^{(1,I,i0)}_{j1}+p^{(0,I,j_{0}0)}_{i0}\big)\frac{\partial }{\partial \theta^{1,I}}\otimes dz^{i0}\wedge dz^{j1}\\
&+\Sigma_{j,j'\in \overline{\mathfrak{I}^{1}}}\big(p^{(1,I,j'1)}_{j1}-p^{(0,I,j1)}_{j'1}\big)\frac{\partial }{\partial \theta^{1,I}}\otimes dz^{j1}\wedge dz^{j'1}\\
&+2\Sigma_{j\in \mathfrak{I}^{1},j'\in\overline{\mathfrak{I}^{1}}}\big(p^{(1,I,j'1)}_{j1}+p^{(0,I,j_{0}0)}_{j'1}\big)\frac{\partial }{\partial \theta^{1,I}}\otimes dz^{j1}\wedge dz^{j'1}\\
&+\Sigma_{j,j'\in \mathfrak{I}^{1}}\big(p^{(0,I,j_{0}0)}_{j'1}-p^{(0,I,j'_{0}0)}_{j1}\big)\frac{\partial }{\partial \theta^{1,I}}\otimes dz^{j1}\wedge dz^{j'1}.
\end{split}
\end{equation}
Besides, always following the page 138 of \cite{BCGGG}, we have to calculate the element $c\in \mathcal{I}^{*}\otimes \bigwedge^{2}(\mathcal{J}/\mathcal{I})$ given by the values $c^{\bullet}_{\bullet,\bullet}$ in \ref{W7} and \ref{W7b}. We obtain 
\begin{equation}
\begin{split}
c&=\frac{\partial }{\partial \theta^{1,I}}\otimes \Sigma_{j\in \mathfrak{I}^{1}}d\varphi^{j}_{I}\wedge  dz^{j1}\\
&=\frac{\partial }{\partial \theta^{1,I}}\otimes \Sigma_{j\in \mathfrak{I}^{1}}\Big[\varphi^{j}_{I,i'0}dz^{i'0}\wedge dz^{j1}+\varphi^{j}_{I,j'1}dz^{j'i}\wedge dz^{j1}\Big].
\end{split}
\end{equation}
Now, the torsion of \ref{W8} vanishes if and only if there exists $p$ satisfying $\overline{\pi}(p)=c.$
It is easy to write this condition because $\overline{\pi}(p)$ and $c$ are expressed in the same base of 
$\mathcal{I}^{*}\otimes \bigwedge^{2}(\mathcal{J}/\mathcal{I}).$ We obtain the conditions:\\
if $J\subset I,$ 
\begin{equation}
\begin{cases}
p^{(0,J,i'0)}_{i0}-p^{(0,J,i0)}_{i'0}=0\\
p^{(0,J,j1)}_{i0}-p^{(0,J,i0)}_{j1}=0\\
p^{(0,J,j'1)}_{j1}-p^{(0,J,j1)}_{j'1}=0,
\end{cases}
\end{equation}
if $J\subset I$ and $J\neq I,$
\begin{equation}
\begin{cases}
p^{(1,J,i'0)}_{i0}-p^{(1,J,i0)}_{i'0}=0\\
p^{(1,J,j1)}_{i0}-p^{(1,J,i0)}_{j1}=0\\
p^{(1,J,j'1)}_{j1}-p^{(1,J,j1)}_{j'1}=0,
\end{cases}
\end{equation}
and 
\begin{equation}\label{W9}
\begin{cases}
p^{(1,I,i'0)}_{i0}-p^{(1,I,i0)}_{i'0}=0\\
p^{(1,I,j1)}_{i0}-p^{(1,I,i0)}_{j1}=0\quad if\;\;j\in \overline{\mathfrak{I}^{1}}\\
-\big(p^{(1,I,i0)}_{j1}+p^{(0,I,j_{0}0)}_{i0}\big)=\frac{1}{2}\varphi^{j}_{I,i0}\quad if\;\; j\in \mathfrak{I}^{1}\\
p^{(1,I,j'1)}_{j1}-p^{(1,I,j1)}_{j'1}=0\quad if\;\; j\;\; and\;\;j'\in \overline{\mathfrak{I}^{1}}\\
p^{(1,I,j'1)}_{j1}+p^{(0,I,j_{0}0)}_{j'1}=\frac{1}{2}\varphi^{j}_{I,j'1}\quad if\;\; j\in \mathfrak{I}^{1}
\;\;and\;\;j'\in\overline{\mathfrak{I}^{1}}\\
p^{(0,I,j_{0}0)}_{j'1}-p^{(0,I,j'_{0}0)}_{j1}=\varphi^{j}_{I,j'1}-\varphi^{j'}_{I,j1}\quad if\;\; j\;\;and\;\; j'\in \mathfrak{I}^{1}.
\end{cases}
\end{equation}
The two first systems are easily satisfied. Also, the five first equations of the last system. The third gives
\begin{equation}
p^{(1,I,i0)}_{j1}=-p^{(0,I,j_{0}0)}_{i0}-\frac{1}{2}\varphi^{j}_{I,i0}\quad if\;\;j\in\mathfrak{I}^{1},
\end{equation}
and the last but one 
\begin{equation}
p^{(1,I,j'1)}_{j1}=\frac{1}{2}\varphi^{j}_{I,j'1}-p^{(0,I,j_{0}0)}_{j'1}\quad if\;\;j\in\mathfrak{I}^{1}
\;\; and\;\;j'\in\overline{\mathfrak{I}^{1}}. 
\end{equation}
Now, it remains the last equation. First, if $I=\{k_{0}0\}$ with $k\in\mathfrak{I}^{1},$ it gives 
\begin{equation}\label{W10}
\begin{split}
p^{(0,k_{0}0,j_{0}0)}_{j'1}&=\varphi^{j}_{k_{0}0,j'1}-\varphi^{j'}_{k_{0}0,j1}
+p^{(0,k_{0}0,j'_{0}0)}_{j1}\\
&=\varphi^{j}_{k_{0}0,j'1}-\varphi^{j'}_{k_{0}0,j1}+\varphi^{k}_{j'_{0}0,j1}-\varphi^{j}_{j'_{0}0,k1}
+p^{(0,j_{0}0,j'_{0}0)}_{k1}\\
&=\varphi^{k}_{j_{0}0,j'1}-\varphi^{j'}_{j_{0}0,k1}+p^{(0,j_{0}0,j'_{0}0)}_{k1}.
\end{split}
\end{equation}
After simplification between the two last lines, we have, if $j,\;j',\;k\in \mathfrak{I}^{1},$
\begin{equation}\label{W12}
\big(\varphi^{j}_{k_{0}0,j'1}-\varphi^{j}_{j'_{0}0,k1}\big)+\big(\varphi^{j'}_{j_{0}0,k1}-
\varphi^{j'}_{k_{0}0,j1}\big)+(\varphi^{k}_{j'_{0}0,j1}-\varphi^{k}_{j_{0}0,j'1}\big)=0.
\end{equation}
Conversely, if this condition is satisfied, it is possible to find $p^{(0,j_{0}0,k_{0}0)}_{j'1}$, symmetric in $j,k,$  verifying the first line of \ref{W10}, that is to say, the last line of \ref{W9}.\\
If, now, $I=\{i0\}$ with $i\in\overline{\mathfrak{I}^{0}},$ (i.e. there is no $i_{1}\in\{1,...,m\}$ verifying 
$(i,i_{1})\in\mathfrak{I}),$ then, the last line of \ref{W9} says 
\begin{equation}
p^{(0,i0,j_{0}0)}_{j'1}=\varphi^{j}_{i0,j'1}-\varphi^{j'}_{i0,j1}+p^{(0,i0,j'_{0}0)}_{j1},
\end{equation}
and this does not implies constraint. \\
At last, if $I=\{j1\},$ we do not have any constraint, even if $j\in\mathfrak{I}^{1}.$\\
Now, we want to look at the case where $I$ contains more than one only element. If $I$ does not contain any element of $\mathfrak{I}^{0},$ there is no constraint. But, if $I$ contains an element $k_{0}0$ with  $k_{0}\in\mathfrak{I}^{0},$ that is to say $\exists k\in\mathfrak{I}^{1}$ such that $(k_{0},k)\in\mathfrak{I}.$ Then, from \ref{W9}, 
\begin{equation}
\begin{split}
p^{(0,I,j_{0}0)}_{j'1}&=\varphi^{j}_{I,j'1}-\varphi^{j'}_{I,j1}+p^{(0,I,j'_{0}0)}_{j1}\\
&=\varphi^{j}_{I,j'1}-\varphi^{j'}_{I,j1}+\varphi^{k}_{I-k_{0},j'_{0}0,j1}-\varphi^{j}_{I-k_{0},j'_{0}0,k1}+p^{(0,I-k_{0},j'_{0}0,j_{0}0)}_{k1}\\
&=\varphi^{k}_{I-k_{0},j_{0}0,j'1}-\varphi^{j'}_{I-k_{0},j_{0}0,k1}+p^{(0,I-k_{0},j'_{0}0,j_{0}0)}_{k1}.
\end{split}
\end{equation}
Simplifying the two last lines, we obtain
\begin{equation}
\big(\varphi^{j}_{I,j'1}-\varphi^{j}_{I-k_{0},j'_{0}0,k1}\big)+\big(\varphi^{j'}_{I-k_{0},j_{0}0,k1}-\varphi^{j'}_{I,j1}\big)+\big(\varphi^{k}_{I-k_{0},j'_{0}0,j1}-\varphi^{k}_{I-k_{0},j_{0}0,j'1}\big)=0,
\end{equation}
or again 
\begin{equation}
\frac{\partial}{\partial (I-k_{0})}\Big[\big(\varphi^{j}_{k_{0}0,j'1}-\varphi^{j}_{j'_{0}0,k1}\big)+\big(\varphi^{j'}_{j_{0}0,k1}
-\varphi^{j'}_{k_{0}0,j1}\big)+\big(\varphi^{k}_{j'_{0}0,j1}-\varphi^{k}_{j_{0}0,j'1}\big)\Big]=0.
\end{equation}
In the brackets $\big[.\big]$ we have the quantity \ref{W12} which is zero. So, we have no new condition. The condition \ref{W12} is the only condition for the system \ref{W8} having no torsion. \\
\\
\\
\\
\\
Here, we want to calculate the Hilbert-Poincar\'e series of the previous system. As \cite{BCGGG}, we denoted by $A^q$, the set of homogeneous solutions of degree $q+1$ to the homogeneous PDE system deduced from $4$. We are able now to recall the definition of the Hilbert-Poincar\'e series
\begin{defn}For a linear PDE system with constant coefficients, the Hilbert-Poincar\'e series is $\sum_qdim(A^q)z^q$ which is defined on the disk of radius $1$.\end{defn} 
Moreover, by general results, we know
\begin{theorem}
The Hilbert-Poincar\'e series is a rational function.
\end{theorem}
By rearranging the variables $z^{i,0},\;z^{j,1},$ if $card(\mathfrak{I})=t,$ we may suppose that\\ 
$\mathfrak{I}=\{(n-k,m-k): k=0,1,...,t-1\}.$ \\
Then, the system \ref{W4} may be written 
\begin{equation}
p^{0}_{n-k,0}+p^{1}_{m-k,1}=\varphi^{k},\;\;\;\forall k=0,...,t-1.
\end{equation}
We have
\begin{equation}
A^{(q)}=\Big\{f=(f_{0},f_{1}): f_{j}=\Sigma_{\mid I\mid =q+1}A^{j}_{I}z^{I}:\frac{\partial f_{0}}{\partial z^{n-k,0}}+\frac{\partial f_{1}}{\partial z^{m-k,1}}=0,\;k=0,...,t-1\Big\}.
\end{equation}
Sometimes, we shall note the variables\\
 $(z^{1,0},\;z^{2,0},...,z^{n,0},\;z^{1,1},...,z^{m,1})=(z^{1},z^{2},...,z^{n},z^{n+1},...,z^{m+n}),$ and
the multi-index $I$ will be note 
$I=(i_{1},i_{2},...,i_{q+1})$ with $i_{j}\in \{(1,0),\;(2,0),...,(n,0),(1,1),...,(m,1)\}=\{1,2,...,n,n+1,...,n+m\}$ or $I=[l^{I}_{1},l^{I}_{2},...,l^{I}_{m+n}]=[l^{I}_{1,0},l^{I}_{2,0},...,l^{I}_{m,1}]$ where $l^{I}_{i,0}=l^{I}_{i}$ is the number of $i=(i,0)$ in $I,$ and $l^{I}_{n+j}=l^{I}_{j,1}$ is the number of $n+j=(j,1)$ in $I.$\\
The above-mentioned condition on $f$ may be written
\begin{equation}
\Sigma_{\mid I\mid=q+1}l^{I}_{n-k,0}A^{0}_{I}z^{I-(n-k,0)}+l^{I}_{m-k,1}A^{1}_{I}z^{I-(m-k,1)}=0,\quad \forall k=0,...,t-1,
\end{equation}
that is to say, for all multi-index $J$ such that $\mid J\mid=q,$ and all $k=0,...,t,$

\begin{equation}
(l^{J}_{n-k,0}+1)A^{0}_{J+(n-k,0)}+(l^{J}_{m-k,1}+1)A^{1}_{J+(m-k,1)}=0,
\end{equation}
or
\begin{equation}\label{W24}
A^{1}_{J+(m-k,1)}=-\frac{(l^{J}_{n-k,0}+1)A^{0}_{J+(n-k,0)}}{l^{J}_{m-k,1}+1}.
\end{equation}

Therefore, if the quantities $A^{0}_{I}$ are knonwn, then the quantities $A^{1}_{I}$ also, except when  $I\cap \mathfrak{J}^{1}=\varnothing$ where $\mathfrak{J}=\{(m-t+1,1),(m-t+2,1),...,(m,1)\}=\{n+m-t+1,n+m-t+2,...,n+m\}.$\\
But, the quantities $A^{0}_{I}$ have to verify another condition. If $J'$ is a multi-index such that $\mid J'\mid=q-1,$ and $k_{1},k_{2}=0,1,...,t-1,$ then, by \ref{W24},
\begin{equation}
\begin{split}
A^{1}_{J'+(m-k_{1},1)+(m-k_{2},1)}&=-\frac{l^{J'+(m-k_{2},1)}_{(n-k_{1},0)}+1}{l^{J'+(m-k_{2},1)}_{(m-k_{1},1)}+1}A^{0}_{J'+(m-k_{2},1)+(n-k_{1},0)}\\
&=-\frac{l^{J'+(m-k_{1},1)}_{(n-k_{2},0)}+1}{l^{J'+(m-k_{1},1)}_{(m-k_{2},1)}+1}A^{0}_{J'+(m-k_{1},1)+(n-k_{2},0)},
\end{split}
\end{equation}

and therefore,
\begin{equation}
\begin{split}
\Big(l^{J'+(m-k_{2},1)}_{(n-k_{1},0)}+1\Big)&\Big(l^{J'+(m-k_{1},1)}_{(n-k_{2},1)}+1\Big)A^{0}_{J'+(m-k_{2},1)+(n-k_{1},0)}\\
&=\Big(l^{J'+(m-k_{1},1)}_{(n-k_{2},0)}+1\Big)\Big(l^{J'+(m-k_{2},1)}_{(n-k_{1},1)}+1\Big)A^{0}_{J'+(m-k_{1},1)+(n-k_{2},0)}
\end{split}
\end{equation}
So, except for a multiplicative constant, in this equality, we can interchange $k_{1}$ and $k_{2}.$ \\
Now, if $I=I'_{0}+I''_{0}+I'_{1}+I''_{1}$ with $I'_{0}\subset \{(1,0),(2,0),...,(n-t,0)\},\quad I''_{0}\subset\{(n-k,0),\;k=0,...,t-1\},\quad\\
 I'_{1}\subset\{(1,1),...,(m-t,1)\},\quad I''_{1}\subset\{(m-k,1),\;k=0,...,t-1\},$ then, to define $I''_{0}+I''_{1},$ with, for example, $\mid I''_{0}\mid+\mid I''_{1}\mid=s,$ it suffices to give the numbers $k_{1},...,k_{s}$ with $k_{j}\in \{0,...,t-1\},$ and then, these $k_{j}$ been interchangeables, we have to affect $0$ to some, and $1$ to the others, which we have $s+1$ ways to do. \\
We have $C^{t-1}_{t-1+s}$ ways to choose $k_{1},...,k_{s}$ and, therefore, $(s+1)C^{t-1}_{t-1+s}$ manners to choose $I''_{0}+I''_{1}.$ \\
We then have $C^{m+n-2t-1}_{q+m+n-2t-s}$ choices to define $I'_{0}+I'_{1}$ when $\mid I'_{0}\mid+\mid I'_{1}\mid=q+1-s.$ \\
At last, we have $(s+1)C^{t-1}_{t-1+s}C^{m+n-2t-1}_{q+m+n-2t-s}$ manners to choose $A^{0}_{I}$ if $\mid I\mid=q+1$ and $\mid I''_{0}+I''_{1}\mid=s.$\\
In the same way, we have $C^{m+n-t-1}_{q+m+n-t}$ choices for $A^{1}_{I}$ if $\mid I\mid=q+1$ and $I\cap \mathfrak{J}^{1}=\varnothing.$\\
For the following calculations, we need a numeric lemma. 
\begin{lemma}
\begin{equation}
\begin{split}
&\Sigma_{s=0}^{q}C^{a}_{a+s}=C^{a+1}_{a+q+1}\\
&\Sigma_{s=p}^{q}C^{a}_{a+s}=C^{a+1}_{a+q+1}-C^{a+1}_{a+p}\\
&\Sigma_{s=0}^{q}C^{b}_{a+s}=C^{b+1}_{a+q+1}-C^{b+1}_{a}\\
&\Sigma_{s=0}^{q}(s+1)C^{a}_{a+s}=(a+1)C^{a+2}_{a+q+1}+C^{a+1}_{a+q+1}\\
&\Sigma_{s=0}^{q}(s+1)C^{b}_{a+s}=(b+1)[C^{b+2}_{a+q+1}-C^{b+2}_{a}]+(b-a+1)[C^{b+1}_{a+q+1}-C^{b+1}_{a}]\;\;if\;\;a\geq b+2\\
&\Sigma_{s=0}^{d}C^{a}_{a+s}C^{b-d}_{b-s}=C^{a+b+1-d}_{a+b+1}\;\;if\;\;d\leq b.
\end{split}
\end{equation}
\end{lemma}
The proofs are elementary. We only write the last one. 

\begin{equation}
\Sigma_{s=k}^{d}C^{b-d}_{b-s}=\Sigma_{s=k}^{d}C^{b-d}_{b-d+(d-s)}=\Sigma_{s'=0}^{d-k}C^{b-d}_{b-d+s'}=C^{b-d+1}_{b+1-k},
\end{equation}
so
\begin{equation}
\begin{split}
\Sigma_{s=0}^{d}C^{a}_{a+s}C^{b-d}_{b-s}&=\Sigma_{s=0}^{d}C^{b-d}_{b-s}\Sigma_{k=0}^{s}C^{a-1}_{a-1+k}=\Sigma_{k=0}^{d}\Sigma_{s=k}^{d}C^{a-1}_{a-1+k}C^{b-d}_{b-s}\\
&=\Sigma_{k=0}^{d}C^{a-1}_{a-1+k}C^{b+1-d}_{b+1-k}=\Sigma_{k=0}^{d}C^{a-2}_{a-2+k}C^{b+2-d}_{b+2-k}=...=\Sigma_{s=0}^{d}C^{0}_{0+s}C^{a+b-d}_{a+b-s}\\
&=\Sigma_{s'=0}^{d}C^{a+b-d}_{a+b-d+s'}=C^{a+b-d+1}_{a+b+1}.
\end{split}
\end{equation}
Using this lemma, we obtain the dimension of the space $A^{(q)}$
\begin{equation}
\begin{split}
Dim A^{(q)}&=\Sigma_{s=0}^{q+1}(s+1)C^{t-1}_{t-1+s}C^{m+n-2t-1}_{q+m+m-2t-s}+C^{m+n-t-1}_{q+m+n-t}\\
&=\Sigma_{s=1}^{q+1}sC^{t-1}_{t-1+s}C^{m+n-2t-1}_{q+m+m-2t-s}+\Sigma_{s=0}^{q+1}C^{t-1}_{t-1+s}C^{m+n-2t-1}_{q+m+m-2t-s}+C^{m+n-t-1}_{q+m+n-t}\\
&=t\Sigma_{s=1}^{q+1}C^{t}_{t-1+s}C^{m+n-2t-1}_{q+m+m-2t-s}+C^{m+n-t-1}_{q+m+n-t}+C^{m+n-t-1}_{q+m+n-t}\\
&=tC^{m+n-t}_{q+m+n-t}+2C^{m+n-t-1}_{q+m+n-t}.
\end{split}
\end{equation}

\section{Torsion's system of the $1$-Cauchy-Fueter equation}
As we saw in the introduction, the second step of the $1$-Cauchy-Fueter complex involved the non-homogeneous torsion's equations of the $1$-Cauchy-Fueter equations :
\begin{equation}\label{w9b}
{\partial^{2}\Phi_{k}\over \partial z^{i1}\partial z^{\theta 0}}-{\partial^{2}\Phi_{k}\over \partial z^{i0}\partial z^{\theta 1}}+{\partial^{2}\Phi_{i}\over \partial z^{k0}\partial z^{\theta 1}}-{\partial^{2}\Phi_{i}\over \partial z^{k1}\partial z^{\theta 0}}+{\partial^{2}\Phi_{\theta}\over \partial z^{i0}\partial z^{k 1}}-{\partial^{2}\Phi_{\theta}\over \partial z^{i1}\partial z^{k 0}}=\varphi_{i\theta k},
\end{equation}
 for all $i,\theta,k$ dans $\{0,1,2,3\}$.
It is easy to see that the left hand term is antisymmetric in $(i,\theta,k)$, so $\varphi_{i\theta k}$  must to be $\C$-analytic in $z=(z^{ij})$ and antisymmetric in $(i,\theta,k)$ therefore gives an element  of $\Lambda^3 (\C[[Z]])^4$. In the following it will be clear that this condition is not sufficient to solve the previous system. The linear system defining the torsion is given by 
\begin{equation}\label{W8b}
\begin{cases}
p^{k}_{i1\theta 0}-p^{k}_{i0\theta 1}+p^{i}_{k0\theta 1}-p^{i}_{k1\theta 0}+p^{\theta}_{i0k1}-p^{\theta}_{i1k0}=\varphi_{i\theta k}\\
p^{k}_{i1\theta 0l0}-p^{k}_{i0\theta 1l0}+p^{i}_{k0\theta 1l0}-p^{i}_{k1\theta 0l0}+p^{\theta}_{i0k1l0}-p^{\theta}_{i1k0l0}={\partial \varphi_{i\theta k}\over \partial z^{l0}}\\ 
p^{k}_{i1\theta 0l1}-p^{k}_{i0\theta 1l1}+p^{i}_{k0\theta 1l1}-p^{i}_{k1\theta 0l1}+p^{\theta}_{i0k1l1}-p^{\theta}_{i1k0l1}={\partial \varphi_{i\theta k}\over \partial z^{l1}},\end{cases}
\end{equation}
where $p^{k}_{ijlq}$ are symmetric by interchanging the pairs $ij$ and $lq$ and $p^{k}_{ijlqpr}$ are symmetric by interchanging  the pairs $ij$, $lq$ and $qr$. The terms at the right and left hand of the equality are antisymmetric with respect to $i,\theta,k$ so it is enough to solve the last two equations with $i<\theta<k$. Consider the form $f=\sum_{i<\theta< k, l}{\partial \varphi_{i\theta k}\over \partial z^{l0}}X^{l}dX^{i}\wedge dX^{\theta}\wedge dX^{k}$ and suppose that we can find a 2-form, $u$, with homogeneous symmetric polynomials of degree 2 as coefficients  : $u=\sum_{i,k,l,\theta}a_{ki\theta l}X^lX^{\theta}dX^{i}\wedge dX^{k}$ such that $du=f$ then $p^{k}_{i1\theta 0 l0}:=a_{ki\theta l}$ solve the second line of equations of \ref{W8b}. On the other hand if we have solutions of the equations, we have a solution of $du=f$. By classical results, this is possible if and only if $df=0$. These conditions give : 
\begin{equation}\label{W11}
{\partial \varphi_{i\theta k}\over \partial z^{l0}}-{\partial \varphi_{l\theta k}\over \partial z^{i0}}+{\partial \varphi_{lik}\over \partial z^{\theta 0}}-{\partial \varphi_{li\theta}\over \partial z^{k0}}=0,
\end{equation}for all $i,\theta,k, l\in\{0,1,2,3\}.$ We can do the same thing with the third line equations \ref{W8b} and we obtain the condition
\begin{equation}\label{W12b}
{\partial \varphi_{i\theta k}\over \partial z^{l1}}-{\partial \varphi_{l\theta k}\over \partial z^{i1}}+{\partial \varphi_{lik}\over \partial z^{\theta 1}}-{\partial \varphi_{li\theta}\over \partial z^{k1}}=0.
\end{equation} It is easy to see that the last equations are antisymmetric in $i,\theta,k, l$.
\medskip

The calculus of the torsion for the prolongation system is more technical. So we need the following lemma:
\begin{lemma}\label{W22}
Let $J=(j_1,\cdots,j_l)$ and $\Lambda=(\lambda_1,\cdots,\lambda_{l^{'}})$ two multi-index We denote by $J^{'}$, $J^{''}$, $J^{'''}$, $J/\{j_l\}$, $J/\{j_{l-1},j_l\}$, $J/\{j_{l-2},j_{l-1},j_l\}$ respectively. If $X_{J}$ are numbers indexed by $J$ and furthermore if these numbers are invariant by permutation of two elements of $J$, we write $X_{(J)}$. Now suppose that we have the identity between the two following forms
\begin{equation}\label{W16}
\begin{split}
\sum_{j_l<j_{l-1}<k}\big [& (X^{k}_{(J^{'})j_l(\Lambda)}-X^{j_l}_{(J^{'})k(\Lambda)})-   (X^{k}_{(J^{''}j_l)j_{l-1}(\Lambda)}-
 X^{j_{l-1}}_{(J^{''}j_l)k(\Lambda)})\\+  
 & (X^{j_l}_{(J^{''}k)j_{l-1}(\Lambda)}-X^{j_{l-1}}_{(J^{''}k)j_l(\Lambda)})\big ]
 X^{(\Lambda)}dX^{j_l}\wedge dX^{j_{l-1}}\wedge dX^{k}\\
=\sum_{j_l<j_{l-1}<k}\big [& (X^{k}_{(J^{''})j_{l-1}(j_l\Lambda)}-X^{j_{l-1}}_{(J^{''})k(j_{l}\Lambda)})-  (X^{k}_{(J^{''})j_{l}(j_{l-1}\Lambda)}- 
X^{j_{l}}_{(J^{''})k(j_{l-1}\Lambda)})\\+ 
& (X^{j_{l-1}}_{(J^{''})j_{l}(k\Lambda)}-X^{j_{l}}_{(J^{''})j_{l-1}(k\Lambda)})\big ]X^{(\Lambda)}dX^{j_l}\wedge dX^{j_{l-1}}\wedge dX^{k}.
\end{split}
\end{equation}
Then the form
\begin{equation}
\begin{split}
\sum_{j_{l-2}<j_{l-1}<k}\big [& (X^{k}_{(J^{''})j_{l-1}(j_l\Lambda)}-X^{j_{l-1}}_{(J^{''})k(j_{l}\Lambda)})-  (X^{k}_{(J^{'''}j_{l-1})j_{l-2}(j_{l}\Lambda)}- 
X^{j_{l-2}}_{(J^{'''}j_{l-1})k(j_{l}\Lambda)})\\+ 
& (X^{j_{l-1}}_{(J^{'''}k)j_{l-2}(j_l\Lambda)}-X^{j_{l-2}}_{(J^{'''}k)j_{l-1}(j_l\Lambda)})\big ]X^{(j_l\Lambda)}dX^{j_{l-2}}\wedge dX^{j_{l-1}}\wedge dX^{k}
\end{split}
\end{equation}
is $d$-closed.
\end{lemma}
\begin{rem}The two forms in \ref{W16} are equal if and only if the form at left hand is $d$-closed.
\end{rem} 

\begin{proof} By elementary but tedious calculus, it is easy to check that the coefficients of the exterior derivative of the form defined in $45$ is exactly the coefficients of the exterior derivative of this form 
\begin{equation}
\begin{split}
\sum_{j_l<j_{l-1}<k}\big [& (X^{k}_{(J^{'})j_l(\Lambda)}-X^{j_l}_{(J^{'})k(\Lambda)})-   (X^{k}_{(J^{''}j_l)j_{l-1}(\Lambda)}-
 X^{j_{l-1}}_{(J^{''}j_l)k(\Lambda)})\\+  
 & (X^{j_l}_{(J^{''}k)j_{l-1}(\Lambda)}-X^{j_{l-1}}_{(J^{''}k)j_l(\Lambda)})\big ]
 X^{(J'')}dX^{j_l}\wedge dX^{j_{l-1}}\wedge dX^{k}.
 \end{split}
 \end{equation}
 On the other hand it is easy to see that the last form is equal to 
 \begin{equation}
 \begin{split}
d\big [ \sum_{j_{l-1}<k}(X^{k}_{(J^{'})j_l(\Lambda)}-X^{j_l}_{(J^{'})k(\Lambda)})
X^{(J')}dX^{j_{l}}\wedge dX^{k}\big ]
\end{split}
\end{equation} 
and so all the previous coefficients are zero.

\end{proof}

To compute the torsion of the prolongation of the system \ref{W8b}, we have essentially to solve the following equation with the given symmetric properties respect to the pairs of index for $p^a$:
\begin{equation}\label{W9c}
p^{k}_{i1\theta 0J1\Lambda 0}-p^{k}_{i0\theta 1J1\Lambda 0}+p^{i}_{k0\theta 1J1\Lambda0}-p^{i}_{k1\theta 0J1\Lambda0}+p^{\theta}_{i0k1J1\Lambda0}-p^{\theta}_{i1k0J_1\Lambda0}={\partial \varphi_{i\theta k}\over \partial z^{J1}\partial z^{\Lambda 0}}
\end{equation}
where $j_11\cdots j_l1$ and  $\lambda_10\cdots \lambda_m 0$ denoted by $J1$ and $\Lambda 0$.
\begin{rem}Recall that the torsion for the first prolongation system defined by $41$ is exactly the compatibility conditions to have integral element for this system. We know that the torsion for the initial system is exactly done by \ref{W11} and \ref{W12b} and so we have just to verify that the system \ref{W9c} has solutions under these assumptions. 
\end{rem}
 Going through the algebraization of the problem, we have to find numbers indexed by $J$ and $\Lambda$, with a appropriate properties of symmetry,  which satisfying:
\begin{equation}\label{W10b}
(Y^{k}_{(Ji)(\Lambda\theta)}-Y^{i}_{(Jk)(\Lambda\theta)})-(Y^{k}_{(J\theta)(\Lambda i)}-Y^{\theta}_{(Jk)(\Lambda i)})+(Y^{i}_{(J\theta)(\Lambda k)}-Y^{\theta}_{(Ji)(\Lambda k)})=Z^{i\theta k}_{(J)(\Lambda)}
\end{equation}
where $Z^{i\theta k}_{(J)(\Lambda)}={\partial \varphi_{i\theta k}\over \partial z^{J1}\partial z^{\Lambda 0}}$ is antisymmetric in $i,\theta, k$ and the parenthesis point out the symmetry in the multi-indices. The form
\begin{equation}\label{W13}
\sum_{i<\theta<k}Z^{i\theta k}_{(J)(\Lambda)}X^{(\Lambda)}dX^{i}\wedge dX^{\theta}\wedge dX^k
\end{equation}
is $d$-closed if \ref{W11} and \ref{W12b} are satified (it suffices to remark that the $d$ of this form is the derivatives with respect to $z^{\Lambda^{'}0}$ and $z^{J1}$ of \ref{W11} with $\lambda_m$ instead of $l$), so for fixing $J$, we can solve \ref{W10b} but perhaps without the symmetry with respect to $J$. Indeed we can symmetrise  with respect to $J$ and obtain finally:
\begin{equation}\label{W14}
(Y^{k}_{(J)i(\Lambda\theta)}-Y^{i}_{(J)k(\Lambda\theta)})-(Y^{k}_{(J)\theta(\Lambda i)}-Y^{\theta}_{(J)k(\Lambda i)})+(Y^{i}_{(J)\theta(\Lambda k)}-Y^{\theta}_{(J)i(\Lambda k)})=Z^{i\theta k}_{(J)(\Lambda)}.
\end{equation}
We point out  here that \ref{W14} has a solution if and only if \ref{W13} is $d$-closed and so it is a necessary condition to solve \ref{W10b}. The equation \ref{W14} is the first step of the construction now we have to obtain one more symmetry between $J$ and $i$, $J$ and $\theta$, $J$ and $k$. All the solution of \ref{W14} are deduced by the sum of the previous one and the following term: $Y^{k}_{(J)i(\Lambda \theta)}+X^{k}_{(J)(i\Lambda\theta)}$. We want to choose $X^{k}_{(J)(i\Lambda\theta)}$ such that 
\begin{equation}\label{W15}
Y^{k}_{(J)i(\Lambda \theta)}+X^{k}_{(J)(i\Lambda\theta)}-Y^{i}_{(J)k(\Lambda \theta)}-X^{i}_{(J)(k\Lambda\theta)}=Y^{k}_{(Ji)(\Lambda\theta)}-Y^{i}_{(Jk)(\Lambda\theta)}.
\end{equation}
It is enough to find $X^{k}_{(Ji)\Lambda\theta}$ which satisfy \ref{W15} and symmetrise with respect to $\Lambda\theta$. A necessary and sufficient condition  to get \ref{W15} is the following:
\begin{equation}\label{W17}
\sum_{i<k}\big ((Y^{k}_{(J)i(\Lambda \theta)}+X^{k}_{(J)(i\Lambda\theta)})-(Y^{i}_{(J)k(\Lambda \theta)}+X^{i}_{(J)(k\Lambda\theta)})\big )X^{(J)}dX^{i}\wedge dX^{k}=d\big (\sum Y^{k}_{(Ji)(\Lambda\theta)}X^{(Ji)}dX^{k}\big )
\end{equation}
which is equivalent to
\begin{equation}\label{W18}
\begin{split}
& (Y^{k}_{(J)i(\Lambda \theta)}-Y^{i}_{(J)k(\Lambda \theta)})-(Y^{k}_{(J^{'}i)j_l(\Lambda \theta)}-Y^{j_l}_{(J^{'}i)k(\Lambda \theta)})+(Y^{i}_{(J^{'}k)j_l(\Lambda \theta)}-Y^{j_l}_{(J^{'}k)i(\Lambda \theta)})\\
 = & - (X^{k}_{(J)(i\Lambda\theta)}-X^{i}_{(J)(k\Lambda\theta)})+(X^{k}_{(J^{'}i)(j_l\Lambda\theta)}-X^{j_l}_{(J^{'}i)(k\Lambda\theta)})-(X^{i}_{(J^{'}k)(j_l\Lambda\theta)}-X^{j_l}_{(J^{'}k)(i\Lambda\theta)}).
\end{split}
\end{equation}
The form at left hand of the equality is antisymmetric in $(i,j_l,k)$ and the form
\begin{equation}\label{W19}
\begin{split}
\sum_{i<j_l<k}\big ((Y^{k}_{(J)i(\Lambda \theta)}-Y^{i}_{(J)k(\Lambda \theta)})-(Y^{k}_{(J^{'}i)j_l(\Lambda \theta)}-Y^{j_l}_{(J^{'}i)k(\Lambda \theta)}) + & (Y^{i}_{(J^{'}k)j_l(\Lambda \theta)}-Y^{j_l}_{(J^{'}k)i(\Lambda \theta)})\big )\\ & X^{(\Lambda\theta)}dX^{i}\wedge dX^{j_l}\wedge dX^{k}
\end{split}
\end{equation}
is $d$-closed thanks to \ref{W11} and \ref{W12b}. So it is equal to 
\begin{equation}
d\big (\sum_{j_l<k}(X^{k}_{(J^{'})j_l(i\Lambda\theta)}-X^{j_l}_{(J^{'})k(i\Lambda\theta)})X^{(i\Lambda\theta)}
dX^{j_l}\wedge dX^{k}\big ),
\end{equation}
we solve for fixing $J^{'}$ and we symmetrise with respect to. With the help of \ref{W19}, we have \ref{W18} with $X^{k}_{(J^{'})j_l(i\Lambda\theta)}$ instead of $X^{k}_{(J)(i\Lambda\theta)}$. To override this commutation failure, we correct again in the following way: $X^{k}_{(J^{'})j_l(i\Lambda\theta)}+
Z^{k}_{(J^{'})(j_l i\Lambda\theta)}$
such that the form 
\begin{equation}\label{W20}
d\big (\sum_{j_l<k}\big ((X^{k}_{(J^{'})j_l(i\Lambda\theta)}+
Z^{k}_{(J^{'})(j_l i\Lambda\theta)})-(X^{j_l}_{(J^{'})k(i\Lambda\theta)}+
Z^{j_l}_{(J^{'})(k i\Lambda\theta)})\big )X^{(J^{'})}dX^{j_l}\wedge dX^{k}\big)=0.
\end{equation}
Using the same argument as to obtain \ref{W18} and \ref{W19}, to get \ref{W20} the following form must be $d$-closed:
\begin{equation}\label{W21}
\begin{split}
\sum_{j_l<j_{l-1}< k}\big ((& X^{k}_{(J^{'})j_l(i\Lambda\theta)} -X^{j_l}_{(J^{'})k(i\Lambda\theta)})-  (X^{k}_{(J^{''}j_l)
j_{l-1}(i\Lambda\theta)}  -X^{j_{l-1}}_{(J^{''}j_l)k(i\Lambda\theta)}) \\ & + (X^{j_l}_
{(J^{''}k)j_{l-1}(i\Lambda\theta)}-X^{j_{l-1}}_{(J^{''}k)j_l(i\Lambda\theta)})\big )
 X^{(i\Lambda\theta)}dX^{j_{l}}\wedge dX^{j_{l-1}}\wedge dX^{k}.
\end{split}
\end{equation}
Using lemma \ref{W22} with $\tilde{J}=(J,i)$ and $\tilde{\Lambda}=(\Lambda,\theta)$, we obtain without difficulties that \ref{W21} is $d$-closed (it suffices to remark that the $d$ of this form is the derivatives  with respect to $z^{J^{'}1}$ and $z^{\Lambda 0}$ of \ref{W12b} with $j_l$ instead of $l$) and so we get a form $Z^{k}_{(J^{''})j_{l-1}(j_l i\Lambda\theta)}$ instead of $Z^{k}_{(J^{'})(j_l i\Lambda\theta)}$  . We can use this process up to obtain $Z^{k}_{j_1j_2(j_3j_4\cdots j_li\Lambda\theta)}$ and we modify again by a form $Z^{k}_{j_1(j_2\cdots j_li\Lambda\theta)}$. Now the last form has no commutation failure so we get the identity:
\begin{equation}\label{W23}
\begin{split}
& (Z^{k}_{j_1j_2(j_3j_4\cdots j_li\Lambda\theta)}+Z^{k}_{j_1(j_2\cdots j_li\Lambda\theta)})-(Z^{j_2}_{j_1 k(j_3j_4\cdots j_li\Lambda\theta)}+Z^{j_2}_{j_1(k j_3\cdots j_li\Lambda\theta)})\\ &\ \ \ \  =  Z^{k}_{(j_{1}j_{2})(j_3j_4\cdots j_li\Lambda\theta)}-Z^{j_2}_{(j_{1}k)(j_3j_4\cdots j_li\Lambda\theta)}.
\end{split}
\end{equation}
We can modify now $Z^{k}_{(j_1j_2)j_3(j_4\cdots j_li\Lambda\theta)}$ by $Z^{k}_{(j_1j_2)j_3(j_4\cdots j_li\Lambda\theta)}$+$Z^{k}_{(j_1j_2)(j_3j_4\cdots j_li\Lambda\theta)}$ to get a form 

$Z^{k}_{(j_1j_2j_3)(j_4\cdots j_li\Lambda\theta)}$. So we can go back up to the term $Y^{k}_{(J)i(\Lambda\theta)}$ that we will be change by $Y^{k}_{(J)i(\Lambda\theta)}+Z^{k}_{(J)(i\Lambda\theta)}$ such that 
\begin{equation}
\begin{split}
(Y^{k}_{(J)i(\Lambda\theta)}+Z^{k}_{(J)(i\Lambda\theta)})-(Y^{i}_{(J)k(\Lambda\theta)}+Z^{i}_{(J)(k\Lambda\theta)}) 
= Y^{k}_{(Ji)(\Lambda\theta)}-Y^{i}_{(Jk)(\Lambda\theta)}
\end{split}
\end{equation}
and finally $Y^{k}_{(Ji)(\Lambda\theta)}$ is a solution of \ref{W10b}.
\begin{rem}Indeed all arguments above are still available in $\C^{4n}$,  we have the torsion's system in the general case.
\end{rem}

\section{Torsion and Complex associated to the kernel of a partial differential system of order one with constant coefficients}
In this section, we consider a linear homogeneous system of PDE with constant coefficients of order one denoted by $A$: $\sum a^{m}_{ij}P^{i}_{j}=0$ with $1\leq m\leq \alpha$, $1\leq i\leq \beta$, $1\leq j\leq n$ and the standard notations $P^{i}_j:={\partial P^{i}\over \partial x_{j}}$. We recall that $A^0$ is the set of 1-jets solutions of $A$ and  $A^1$ is the set of 2-jets such that $\sum a^{m}_{ij}P^{i}_{lj}=0$ for all $l$. 
\begin{defn} We say that the sequence of $1$-jets $(P^{i}_{1j})_{j\geq 1}, (P^{i}_{2j})_{j\geq 2}, \cdots, (P^{i}_{kj})_{j\geq k}$ is k-regular if and only if $(P^{i}_{1j}) \in A^0$ and 
$$a^{m}_{i1}P^{i}_{1l}+\cdots +a^{m}_{i(l-1)}P^{i}_{(l-1)l}=-\sum_{j\geq l}a^{m}_{ij}P^{i}_{lj},$$
for all $2\leq l\leq k$.
\end{defn}
\begin{rem}We can adapt easily the previous definition for $A$ linear partial differential system with constant coefficients for which the matrix of total symbol contains only homogeneous polynomials of order $\gamma$. The previous definition depends of the coordinates but it becomes coordinates free if we consider only generic coordinates (see \cite{BCGGG}, pp 119), it will be more clear in the following.
\end{rem}
\begin{defn}We say that $A$ is in involution if and only if all $k$-regular sequel can be extended by a $k+1$-regular sequel. More precisely : if $(P^{i}_{1j})_{j\geq 1}, (P^{i}_{2j})_{j\geq 2}, \cdots, (P^{i}_{kj})_{j\geq k}$ is a $k$-regular sequence, there exists $(P^{i}_{(k+1)j})_{j\geq k+1}$ such that $(P^{i}_{1j})_{j\geq 1}, (P^{i}_{2j})_{j\geq 2}, \cdots, (P^{i}_{kj})_{j\geq k},(P^{i}_{(k+1)j})_{j\geq k+1}$ is $k+1$-regular.
\end{defn}
We will see in the next proposition that the involution in the previous sense, is exactly the same than the involution of the tableau associated to $A$ in the sense of Cartan (see the definition below). So generic coordinates for this notion of involution is the same than generic coodinates for Cartan's tableau involution.
\begin{prop}$A$ is in involution if and only if the tableau associated to $A$ is in involution in the sense of Cartan.
\end{prop}

\begin{proof}A tableau is in involution in the sense of Cartan if and only if $dim A^1=dim A^{0}+dim A^{0}_{1}+\cdots +dim A^{0}_{n-1}$ where $A^{0}_{j}$ is the set of one jets in the variables $x_{j+1},\cdots , x_n$ solutions of $A$. If $(P^{i}_{lj})_{l,j\geq 1}$ a $2$-jet is in $A^{1}$ then 
$$a^{m}_{ij}P^{i}_{1j}=0,\ a^{m}_{ij}P^{i}_{2j}=0,\cdots ,\ a^{m}_{ij}P^{i}_{nj}=0,$$
with the usual notation: if an index is repeated then we sum with respect to it.
Using the last equalities, we deduce that $(P^{i}_{lj})\in A^{1}$ implies that $(P^{i}_{1j})$ is in $A^{0}$ and $$a^{m}_{i1}P^{i}_{1l}+\cdots + a^{m}_{i(l-1)}P^{i}_{(l-1)l}$$ is in the image of the endomorphism defined by $\sum_{j\geq l}a^{m}_{ij}P^{i}_{lj}$ denoted by $A^{[0]}_{l-1}$ for all $2\leq l\leq n$. Now it is obvious that always
$$dim A^1\leq dim A^{0}+dim A^{0}_{1}+\cdots +dim A^{0}_{n-1}.$$ On the other hand, the equality $dim A^1=dim A^{0}+dim A^{0}_{1}+\cdots +dim A^{0}_{n-1}$ is  obtained when all $l$-regular sequences of jets, $(P^{i}_{1j})_{j\geq 1}, (P^{i}_{2j})_{j\geq 2}, \cdots, (P^{i}_{lj})_{j\geq l}$, can be extended in a $l+1$-regular sequence of jets for all $2\leq l\leq n-1$.
\end{proof}
\begin{rem}The last proposition says exactly that we can construct all the 2-jets in $A^1$ only with the help of any $1$-jets in $A^{0}$ which is completed like in the previous proposition. Clearly this proposition can be adapted mutatis mutandis if $A^{p}$ is in involution with $p>0$.
\end{rem}
Let us consider, $\mathcal{A}$ and $\mathcal{B}$ two partial differential operators of order one with constant coefficients:
$$\sum a^{m}_{ij}P^{i}_{j},\ 1\leq m\leq \alpha ,\  1\leq i\leq \beta ,\ 1\leq j\leq n\ \ (\mathcal{A}),$$
$$\sum b^{m}_{ij}Q^{i}_{j},\ 1\leq m\leq \gamma,\ 1 \leq i\leq \alpha,\ 1\leq j\leq n\ \ (\mathcal{B}).$$
The operators $\mathcal{A}$ and $\mathcal{B}$ induce two endomorphisms on the sets of one jets which, by abuse of notations, we denote by $\mathcal{A}$ and $\mathcal{B}$ too. Similary the operator $\mathcal{A}$ induces an endomorphism denoted by $\mathcal{A}^{1}$ on $2$-jets which is obviously defined by
$$\sum a^{m}_{ij}P^{i}_{lj}\ \ \ l=1,\cdots,n.$$ We can define $\mathcal{A}^{q}$ in the same way. Suppose that the following sequence of endomorphisms is exact:
$$S^{\beta}_{2n}\overset{\mathcal{A}^{1}}{\rightarrow}S^{\alpha}_{1n}\overset{\mathcal{B}}{\rightarrow}S^{\gamma}_{0n}, $$
where $S^{i}_{ln}$ are the sets of $i$-vectors valued jets of order $l$ in $n$-variables. We want to show that the involution of $A$ (the PDE system associated to $\mathcal{A}$) is a hereditary property by the previous exact sequence. First, we etablish this lemma:
\begin{lemma}Suppose that the tableau associated to $A$ is in involution and 
$$ S^{\beta}_{2n}\overset{\mathcal{A}^{1}}{\rightarrow}S^{\alpha}_{1n}\overset{\mathcal{B}}{\rightarrow}S^{\gamma}_{0n} $$ is an exact sequence, then the tableau associated to $B$ (the PDE system associated to $\mathcal{B}$) is in involution if and only if 
$$dim B^{0}_j=dim S^{\beta}_{2(n-j)}-dim A^{1}_j\ \ \forall 1\leq j\leq n-1.$$
\end{lemma}
\begin{rem}In fact, the operator $\mathcal{B}$ is so called the torsion of the system $A$ because by definition of the torsion, the sequence
$$S^{\beta}_{2n}\overset{\mathcal{A}^{1}}{\rightarrow}S^{\alpha}_{1n}\overset{\mathcal{B}}{\rightarrow}S^{\gamma}_{0n} $$
is exact. But in general, the associated sequence of PDE system is not exact, because the tableau associated to $A$ is not necessary in involution.
\end{rem}
\begin{proof}The tableau $A$ is in involution so a jet in $S^{\alpha}_{2n}$ is in $Im(\mathcal{A}^2)$ if and only if it is in $B^{1}=ker(\mathcal{B}^1)$ and therefore we have the following equality $dim B^{1}=dim S^{\beta}_{3n}-dim A^{2}$. On the other hand, if the tableau associated to $A$ is in involution, $dim A^{2}=dim A^{1}+dim A^{1}_{1}+\cdots +dim A^{1}_{n-1}$. Clearly we have $dim B^{0}=dim ker(\mathcal{B})=dim S^{\beta}_{2n}-dim A^{1}$ and $dim B^{0}_j\geq dim S^{\beta}_{2(n-j)}-dim A^{1}_j$ for $j\geq 1$. If the equalities hold for all $j\geq 1$, we have
$$dim B^{1}=dim S^{\beta}_{3n}-dim S^{\beta}_{2n}-\sum_{j=1}^{n-1}dim S^{\beta}_{2(n-j)}+dim B^{0} +\sum_{j=1}^{n-1}dim B^{0}_j.$$
But elementary calculation gives 
$$dim S^{\beta}_{3n}-dim S^{\beta}_{2n}-\sum_{j=1}^{n-1}dim S^{\beta}_{2(n-j)}=0,$$
and therefore the tableau associated to $B$ is in involution. Conversely if there exists $j$ such that $dim B^{0}_j> dim S^{\beta}_{2(n-j)}-dim A^{1}_j$ then
$$dim B^{1}< dim B^{0} +\sum_{j=1}^{n-1}dim B^{0}_j.$$
\end{proof}
We can now prove the hereditary property for the involution
\begin{prop}Under the assumptions of the previous lemma, the tableau associated to $B$ is in involution.
\end{prop}
 \begin{proof}According to the lemma 4.6, we have to prove that a $1$-jet $(\theta^{m}_{l})\in \mathcal{B}^{0}=ker(\mathcal{B})=Im(\mathcal{A}^{1})$ with $(\theta^{m}_l)=0$ for all $ 1\leq l<k$, is the image of  a $2$-jet in $\mathcal{A}^1$, $\tilde P^{m}_{lj}$, with $\tilde P^{m}_{lj}=0$ for all $1\leq l\ \hbox{and}\ j<k$.
 \bigskip
 
- Suppose $k=2$, we have $\sum a^{m}_{ij}P^{i}_{lj}=\theta^{m}_{l}$ for all $l\geq 2$, $\sum a^{m}_{ij}P^{i}_{1j}=0$, and therefore 
 $$a^{m}_{i1}P^{i}_{21}+\sum_{j\geq 2 }a^{m}_{ij}P^{i}_{2j}=\theta^{m}_2,$$
 which is the same thing, thanks to the commutating properties of $2$-jets
 
 $$a^{m}_{i1}P^{i}_{12}+\sum_{j\geq 2 }a^{m}_{ij}P^{i}_{2j}=\theta^{m}_2.$$
 Using the proposition 4.4 and the definition 4.3, we obtain the existence of a one jet 
 $({}^1\!P^{i}_{2j})_{j\geq 2}$ such that
 $$-\sum_{j\geq 2}a^{m}_{ij}{}^1\!P^{i}_{2j}+\sum_{j\geq 2 }a^{m}_{ij}P^{i}_{2j}=\theta^{m}_2,$$
 that is to say
  $$\sum_{j\geq 2}a^{m}_{ij}\big(-{}^1\!P^{i}_{2j}+P^{i}_{2j}\big)=\theta^{m}_2.$$
So  we put $(\tilde P^{i}_{2j})_{j\geq 2}:=(-{}^1\!P^{i}_{2j}+P^{i}_{2j})_{j\geq 2}$ and we want to construct a $1$-jet $(\tilde P^{m}_{3j})_{j\geq 3}$ with the appropriate commutating properties with respect to $(\tilde P^{m}_{2j})_{j\geq 2}$. We start with the equality 
$$a^{m}_{i1}P^{i}_{31}+\sum_{j\geq 2}a^{m}_{ij}P^{i}_{3j}=\theta^{m}_{3},$$
which can be write obviously
$$a^{m}_{i1}P^{i}_{13}+a^{m}_{i2}{}^1\!P^{i}_{23}-a^{m}_{i2}{}^1\!P^{i}_{23}+\sum_{j\geq 2}a^{m}_{ij}P^{i}_{3j}=\theta^{m}_{3}.$$ With the help of proposition 4.4 and the definition 4.3, we get 
$$-\sum_{j\geq 3}a^{m}_{ij}{}^2\!P^{i}_{3j}-a^{m}_{i2}{}^1\!P^{i}_{23}+a^{m}_{i2}P^{i}_{32}+\sum_{j\geq 3}a^{m}_{ij}P^{i}_{3j}=\theta^{m}_ {3},$$ and finally
$$-a^{m}_{i2}{}^1\!P^{i}_{23}+a^{m}_{i2}P^{i}_{32}+\sum_{j\geq 3}a^{m}_{ij}\big(-{}^2\!P^{i}_{3j}+P^{i}_{3j}\big)=a^{m}_{i2}\tilde P^{i}_{23}+\sum_{j\geq 3}a^{m}_{ij}\big(-{}^2\!P^{i}_{3j}+P^{i}_{3j}\big)=\theta^{m}_3$$ and therefore we have the commutating properties needed, if we put $ (\tilde P^{i}_{3j})_{j\geq 3}:=\big(-{}^2\!P^{i}_{3j}+P^{i}_{3j}\big)_{j\geq 3}$. Suppose that we have choosen in a similar way  $(\tilde P^{i}_{kj})_{j\geq k}:=(-{}^{(k-1)}\!P^{i}_{kj}+P^{i}_{kj})_{j\geq k}$ for all $k\leq l$, we want to construct a $1$-jet $(\tilde P^{i}_{(l+1)j})_{j\geq l+1}$ with the required commutating properties with respect to $(\tilde P^{i}_{kj})_{j\geq k}$ for all $k\leq l$. We start with the equality 
$$a^{m}_{ij}P^{i}_{(l+1)j}=a^{m}_{i1}P^{i}_{1(l+1)}+
\sum_{j=2}^{l}a^{m}_{ij}{}^{(j-1)}\!P^{i}_{j(l+1)}-\sum_{j=2}^{l}a^{m}_{ij}{}^{(j-1)}\!P^{i}_{j(l+1)}+\sum_{j\geq 2}a^{m}_{ij}P^{i}_{4j}=\theta^{m}_{l+1},$$
we use again proposition 4.4 and the definition 4.3 and we get:
$$-\sum_{j\geq l+1}a^{m}_{ij}{}^l\!P^{i}_{(l+1)j}-\sum_{j=2}^{l}a^{m}_{ij}{}^{(j-1)}\!P^{i}_{j(l+1)}+\sum_{j=2}^{l}a^{m}_{ij}P^{i}_{(l+1)j}+\sum_{j\geq l+1}a^{m}_{ij}P^{i}_{(l+1)j}=\theta^{m}_{l+1}$$
which can be written 
$$\sum_{j\geq l+1}a^{m}_{ij}(-{}^l\!P^{i}_{(l+1)j}+P^{i}_{(l+1)j})+\sum_{j=2}^{l}a^{m}_{ij}(-{}^{(j-1)}\!P^{i}_{j(l+1)}+ P^{i}_{j(l+1)})=\theta^{m}_{l+1}$$
and so 
$$\sum_{j\geq l+1}a^{m}_{ij}(-{}^l\!P^{i}_{(l+1)j}+P^{i}_{(l+1)j})
+\sum_{j=2}^{l}a^{m}_{ij}\tilde P^{i}_{j(l+1)}=\theta^{m}_{l+1},$$
therefore we  choose $(\tilde P^{i}_{(l+1)j})_{j\geq l+1}:=(-{}^l\!P^{i}_{(l+1)j}+P^{i}_{(l+1)j})_{j\geq l+1}.$
The proof is complete for a  jet $(\theta^{m}_{l})$ satifying $(\theta^{m}_{1})=0$
\bigskip

- If $k$ is bigger than 2, we have a jet $(\theta^{m}_l)\in  \mathcal{B}^{0}=ker(\mathcal{B})=Im(\mathcal{A}^1)$ with $(\theta^{m}_l)=0$ for all $l<k$ and we want to show that this $1$-jet is the image by $\mathcal{A}^1$ of a $2$-jet $(P^{i})_{lj}$ with $P^{i}_{lj}=0$ for all $1\leq l\ \hbox{and}\ j < k$. We proceed by induction on $k$: the induction hypothesis implies that the restrictions of the endomorphisms, $\mathcal{A}^{1}$ and $\mathcal{B}$, to the plane generated by the variables $x_{k-1},\cdots, x_{n}$ define an exact sequence. Furthermore the restriction operator $A$ to the plane $x_{k-1},\cdots, x_{n}$ is still in involution: the involution property is stable by restriction on plane  generated by $x_{k-1}\cdots x_n$; it is a well known fact (see for example the characterization due to Matsushima of involution in \cite{M} and \cite{BCGGG} pages 119 and 120) but it is a nice exercise to see this with the help of the proposition 4.4. Therefore all the assumptions needed are satisfied to apply the previous case for $k=2$ to the endomorphisms restricted to the plane $x_{k-1},\cdots, x_{n}$.

\end{proof}
 
 In the following, we construct the complex associated to a linear operator differential of order one, $A_0$ with tableau in involution, using the previous result. Let $\mathcal{A}_0=a^{m}_{ij}{\partial P^{i}\over \partial x_j}$ with $1\leq m\leq \alpha ,\  1\leq i\leq \beta ,\ 1\leq j\leq n$. We suppose that the endomorphism induced by $\mathcal{A}_0$ between the space  $S^{\beta}_{1n}$ and $S^{\alpha}_{0n}$ is surjective. The torsion $\mathcal{A}_1$, which is only a representative of the class of equations which define $Im(\mathcal{A}_{1})$ in $S^{\alpha}_{1n}$ with minimal number, define a differential operator of order one denoted by $\mathcal{A}_1$ too. Similary by induction, we define $\mathcal{A}_i$ operators of order one for all $i\in \N$. By the previous result, all the $\mathcal{A}_i$ are in involution and by the Cartan-Kahler theorem we have the exact sequence $S$ (possibly infinite):
 $$\big(C^{w}_{x_0}(\R^n)\big)^{\beta }\overset{\mathcal{A}_0}{\rightarrow}\big(C^{w}_{x_0}(\R^n)\big)^{\alpha}\overset{\mathcal{A}_1}{\rightarrow}\big(C^{w}_{x_0}(\R^n)\big)^{\alpha_1}\cdots \overset{\mathcal{A}_i}{\rightarrow}\cdots $$
 where $\big(C^{w}_{x_0}(\R^n)\big)^{\alpha}$ is an $\alpha$-vector with entries germs in $x_0$ of real analytic functions on $\R^n$.
 \begin{prop}The exact sequence $S$ is finite.
\end{prop}
\begin{rem}Although all the previous facts are elementary, we do not have an elementary proof of this fact. The involution of the operator $\mathcal{A}_i$ implies subtle combinatory properties on the dimension of the tableau associated to $\mathcal{A}_{i}$ which we are not able to treat with simple arguments.

With the help of theorem A in \cite{N}, the above complex is exact on $\mathcal{C}^{\infty}(\Omega)$ with $\Omega$ convex.
\end{rem} 
\begin{proof}By classical results (see for example, \cite{N} theorem A), the previous exact sequence give the exact sequence below:
 $$(\C[X])^{\beta}\overset{\mathcal {}^t\!{A}_0(X)}{\longleftarrow}(\C[X])^{\alpha}\overset{\mathcal {}^t\!{A}_1(X)}{\longleftarrow}(\C[X])^{\alpha_1}\cdots \overset{\mathcal {}^t\!{A}_i(X)}{\longleftarrow}\cdots,$$
 where $\mathcal{A}_i(X)$ is the matrix symbol associated to $\mathcal{A}_i$. This exact sequence give a resolution of finitely generated graded $\C[X]$-module defined by the kernel $\mathcal {}^t\!{A}_0(X)$. The Hilbert Syzygy theorem give a unique finite free resolution of length $l\leq n+1$ up to complexes isomorphism (see \cite{Ei} for the classical facts on Hilbert Syzygy theorem). The matrix $\mathcal {}^t\!{A}_i$ contains only polynomials of degree one, so the above resolution is minimal and finite by Hilbert Syzygy theorem.
 \end{proof}
\begin{rem}The Dolbeault complex is relevable of the previous construction: the Cauchy-Riemann equations are in involution in sense of Cartan (See \cite{BCGGG} pp 155-156). The Cauchy-Fueter complex is particulary interesting because the Cauchy-Fueter equations do not have tableau in involution. So we cannot apply the above proposition and indeed the complex contains an operator of order 2. We are going to develop this example in the next section.
\end{rem}

\section{The Cauchy-Fueter complex}

We begin with the simplest but illuminating example of PDE system with a tableau which is not in involution and so it cannot be treated as before: the Cauchy-Fueter equations in $\R^8$. Using the coordinates $z^{i0},z^{i1}$ as in \cite{WW2}, section 2 and 3 give two operators, $tor_0$ and $tor_1$ (remenber $tor_0$ is a PDE system of order $2$ not of order $1$), such that the following  sequence is exact:
$$(\mathcal{C}^{w}(\R^8))^{2}\overset{CF}{\rightarrow} (\mathcal{C}^{w}(\R^8))^{4}\overset{tor_0}{\rightarrow} \lambda(\R^8,\Lambda^3(\C^4))\overset{tor_1}{\rightarrow} \lambda(\R^8,\C^2\otimes\Lambda^4(\C^4))$$
where $\mathcal{C}^{w}(\R^8)$ are the germs of real analytic functions with values in $\C$, $\lambda(\R^8,\Lambda^3(\C^4))$ are the $3$-forms in $\C^4$ with coefficients in $\mathcal{C}^{w}(\R^8)$ and $\lambda(\R^8,\C^2\otimes\Lambda^4(\C^4))$ are the $2$-vectors with entries $4$-forms with coefficients in $\mathcal{C}^{w}(\R^8)$. Clearly this exact sequence induced an exact sequence of endomorphisms between spaces of jets:
$$S^{2}_{(k+4)8}\overset{CF^{k+3}}{\rightarrow} S^{4}_{(k+3)8}\overset{tor_0^{k+1}}{\rightarrow}S^{4}_{(k+1)8} \overset{tor_1^{k}}{\rightarrow}S^{2}_{k8}.$$ Now using the rank theorem, it is obvious to see that 
$$dim(Im (tor_1^k))=dim (S^{4}_{(k+1)8})-dim (S^{4}_{(k+3)8})+dim (S^{2}_{(k+4)8})-dim (ker(CF^{k+3})),$$ where $dim (ker(CF^{k+3}))$ is nothing else than the dimension of the tableau of order $(k+3)$ associated to the Cauchy-Fueter equations (see section 2). Now using $(39)$ in section $2$ with $m=n=t=4$, we have $dim (CF^{k+3})=4C^{4}_{7+k}+2C^{3}_{k+7}$. On the other hand, $dim(S^{p}_{kn})=pC^{n-1}_{k+n-1}$ and therefore the difference 
$dim(Im (tor_1^k))-dim(S^{2}_{k8})$ is a polynomial of degree $7$ in $k$. Moreover we can prove after elementary calculus that this polynomial is zero for $k=0, 1,\cdots, 6, 7$, therefore this polynomial is $0$ which gives the Cauchy-Fueter complex in $\R^{8}$:
$$(\mathcal{C}^{w}(\R^8))^{2}\overset{CF}{\rightarrow} (\mathcal{C}^{w}(\R^8))^{4}\overset{tor_0}{\rightarrow} \lambda(\R^8,\Lambda^3(\C^4))\overset{tor_1}{\rightarrow} \lambda(\R^8,\C^2\otimes\Lambda^4(\C^4))\rightarrow 0.$$
In $\R^{4n}$ the complex is longer and we need some technical lemmas to construct the torsion of $tor_1$ and so on... Nevertheless, we can do it in the same spirit of the section $3$ but the calculus are tedious and there is no additional ideas, so we do not included the proof in this paper. 


\begin{thebibliography}{<00>}




\bibitem[BCGGG]{BCGGG}
R.L. Bryant, S.S. Chern, R.B. Gardner, H.L. Goldschmidt, P.A. Griffiths, 
\newblock Exterior Differential Systems, 
\newblock Spinger Verlag (1991).

\bibitem[Ei]{Ei}
D.Eisenbud,
\newblock The geometry of syzygies. A second course in commutative algebra and algebraic geometry,
\newblock Graduate Texts in Mathematics, 229. Springer-Verlag, New York, 2005. xvi+243 pp

\bibitem[M]{M}
Y. Matsushima,
\newblock Sur les alg\`ebres de Lie semi-involutives,
\newblock Colloque de topologie de Strasbourg, Universit\'e de Strasbourg, (1954-55).

\bibitem[N]{N}
M.Nacinovich,
\newblock Complex Analysis and complexes of differential operators,
\newblock Complex analysis (Trieste, 1980), pp. 105–195, Lecture Notes in Math., 950, Springer, Berlin-New York, 1982.

\bibitem[WW1]{WW1}
W. Wang,
\newblock On the non-homogeneous Cauchy-Fueter equations and Hartog's phenomenon in several quaternionic variables,
\newblock Journal of Geometry and Physics, 58 (2008), 1203-1210.

\bibitem[WW2]{WW2}
W. Wang, 
\newblock The k-Cauchy-Fueter complex, Penrose transformation and Hartogs phenomenon for quaternionic k-regular functions, 
\newblock Journal of Geometry and Physics, 60 (2010),  513-530.


\end{thebibliography}
\end{document}